\documentclass{amsart}
\usepackage[latin1]{inputenc}
\usepackage{enumerate}
\usepackage{amssymb}

\newtheorem{teo}{Theorem}
\newtheorem{propo}{Proposition}
\newtheorem{lem}{Lemma}
\newtheorem{coro}{Corollary}

\theoremstyle{remark}
\newtheorem{rem}{Remark}

\newcommand{\ff}{\varphi}
\newcommand{\fft}{\tilde{\varphi}}
\newcommand{\Om}{\Omega}
\newcommand{\sig}{\sigma}

\newcommand{\de}{\delta}

\newcommand{\ga}{\gamma}

\newcommand{\eps}{\epsilon}
\newcommand{\al}{\alpha}
\newcommand{\be}{\beta}
\newcommand{\la}{\lambda}

\newcommand{\RR}{\mathbb{R}}

\newcommand{\NN}{\mathbb{N}}

\newcommand{\lap}{\Delta}
\newcommand{\LL}{\mathcal{L}}

\newcommand{\Lloc}{L^1_{\text{loc}}}
\newcommand{\rie}{\mathcal{R}}

\newcommand{\ke}{\mathcal{K}}

\newcommand{\Qt}{\tilde{Q}}
\newcommand{\Pt}{\tilde{P}}
\newcommand{\Omt}{\tilde{\Omega}}
\newcommand{\BB}{\tilde{B}}

\newcommand{\Bt}{\tilde{B}}

\title[Weighted inequalities for commutators\ldots]{Weighted inequalities for commutators of Schr\"odinger-{R}iesz transforms}

\author{B. Bongioanni, E. Harboure and O. Salinas}



\thanks{This research is partially supported by grants from Agencia Nacional de Promoci\'on Cient\'ifica y Tecnol\'ogica (ANPCyT), Consejo Nacional de Investigaciones Cient\'ificas y T\'ecnicas (CONICET) and Universidad Nacional del Litoral (UNL), Argentina.}

\address{Instituto de Matem\'atica Aplicada del Litoral CONICET-UNL, Santa Fe, Argentina.}

\email{bbongio\\ harbour\\ salinas@santafe-conicet.gov.ar}

\date{}

\begin{document}

\begin{abstract}
In this work we obtain weighted $L^p$, $1<p<\infty$, and weak $L\log L$ estimates for the commutator of the Riesz transforms associated to a Schr\"odinger operator $-\lap+V$, where $V$ satisfies some reverse H\"older inequality. The classes of weights as well as the classes of symbols are larger than $A_p$ and $BMO$ corresponding to the classical Riesz transforms.
\end{abstract}

\maketitle

\section{Introduction}

Let $V:\RR^d\mapsto \RR$, $d\ge 3$, be a non-negative locally integrable function that belongs to a reverse-H\"older class $RH_q$ for some exponent $q> d/2$, i.e. there exists a constant $C$ such that
\begin{equation}\label{reverse}
\left(\frac{1}{|B|}\int_B V(y)^q\,dy\right)^{1/q} \ \leq\ \frac{C}{|B|}\int_B V(y)\,dy,
\end{equation}
for every ball $B\subset\RR^d$.

For such a potential $V$ we consider the Schr\"odinger operator
\begin{equation*}
\LL=-\Delta+V,
\end{equation*}
and the associated Riesz Transform vector
\begin{equation*}
\rie = \nabla \LL^{-1/2}.
\end{equation*}

Boundedness results of $\rie$ have been obtained in \cite{shen} by Shen, where he shows that they are bounded on $L^p(\RR^d)$ for $1<p<p_0$, with $p_0$ depending on $q$. When $V\in RH_{q}$ with $q\ge d$, $\rie$ and its adjoint $\rie^*$ are in fact Calder\'on-Zygmund operators (see \cite{shen}).

We denote by $T$ either $\rie$ or $\rie^*$. For some function $b$ we will consider the commutator operator
\begin{equation}
T_bf(x)=T(bf)(x)-b(x)Tf(x),\ \ \ x\in\RR^d.
\end{equation}

It is well known (see \cite{MR0412721}) that for the classical case (that is $V\equiv 0$) the corresponding commutators $T_b$ are of strong type $(p,p)$ for $1<p<\infty$ whenever $b$ belongs to $BMO$. However, for the case we deal with in this article, the operators $\rie$ have better properties related to their decay. This behavior was the key point to get a significant improvement about the commutators $T_b$. In fact, in \cite{BHS_commut}, it was obtained strong type $(p,p)$, $1<p<\infty$, for $b$ in a wider space than $BMO$, that is the space $BMO_\infty(\rho)=\cup_{\theta> 0} BMO_\theta(\rho)$, where for $\theta>0$ the space $BMO_\theta(\rho)$ is the set of locally integrable functions $f$ satisfying
\begin{equation}\label{defBMOtheta}
\frac{1}{|B(x,r)|}\int_{B(x,r)}|b(y)-b_B|\,dy\ \leq C\, \left(1+\frac{r}{\rho(x)}\right)^\theta,
\end{equation}
for all $x\in\RR^d$ and $r>0$, with $b_B=\frac{1}{|B|}\int_B b$. A norm for $b\in BMO_\theta(\rho)$, denoted by $[b]_\theta$, is given by the infimum of the constants in \eqref{defBMOtheta}.

The present article is devoted to obtain weighted boundedness for $T_b$. Once again, the special behavior of $\rie$ allows us to get better results than in the classical case.

Particularly, we get strong $(p,p)$ inequalities for $b\in BMO_\infty(\rho)$ and weights in a class larger than Muckenhoupt's. Such classes already appeared in connection with the $L^p$-boundedness of $\rie$ (see \cite{BHS_schr-w}).

Moreover, we obtain weighted weak type inequalities for $T_b$. Related to this, it is important to remember that weak type $(1,1)$ is not true in the case of classical singular integrals (see \cite{MR1317714}). Nevertheless in that situation we are able to prove an $L\log L$ weak estimate but for $b$ in $BMO_\infty(\rho)$ and weights in a class larger than $A_1$. These results are completely new even in the unweighted case.

In order to get the results for $1<p<\infty$ we use basically the same comparison techniques developed in \cite{BHS_schr-w}. However, this method fails for the extreme case $p=1$, so we adapt the techniques in \cite{PradoPerez-SharpCommut}, based on some appropriate Calder\'on-Zygmund decomposition. Also, since the kernels of $\rie$ may not have point-wise smoothness, we have to work with a H\"ormander type condition instead.

The article is organized as follows. In sections~\ref{sec:fcritic} and \ref{sec:bmo} we review some properties concerning the critical radius function and the space $BMO_\infty(\rho)$. Section~\ref{sec:pesos} is devoted to the class of weights where, in particular, we give a method to construct $A_1^{\infty,\rho}$ weights using a maximal function. In Section~\ref{sec:estK} we collect some estimates of the kernels of the Schr\"odinger-Riesz transforms, including a H\"ormander type inequality, which slightly improves Lemma~4 in \cite{chinos}. The main results concerning the boundedness of the commutators are presented in sections~\ref{sec:lpineq} and \ref{sec:debil}.

In the sequel, when $B=B(x,r)$ and $C>0$, we shall use the notation $CB$, to denote the ball with the same center $x$ and radius $Cr$.

\section{The critical radius function}\label{sec:fcritic}

The notion of locality is given by the critical radius function
\begin{equation}\label{laro}
\rho(x)=\sup\left\{r>0:\ \frac{1}{r^{d-2}}\int_{B(x,r)} V \leq
1\right\}, \ \ \ x\in\RR^d,
\end{equation}
which, under our assumptions, satisfies $0<\rho(x)<\infty$ (see \cite{shen}).

\begin{propo}[\cite{shen}]\label{propo:rho_xy}
If $V\in RH_{d/2}$, there exist $c_0$ and $N_0\ge 1$ such that
\begin{equation}\label{rox_vs_roy}
c_0^{-1} \rho(x)\left(1+\frac{|x-y|}{\rho(x)}\right)^{-N_0} \ \leq \ \rho(y) \ \leq \ c_0\,\rho(x)\left(1+\frac{|x-y|}{\rho(x)}\right)^{\frac{N_0}{N_0+1}},
\end{equation}
for all $x, y \in \RR^d$.
\end{propo}

\begin{coro}\label{coro:dero}
Let $x, y \in B(x_0,R_0)$. Then,
\begin{enumerate}[(i)]
\item There exists $C>0$ such that
\begin{equation}\label{R0y}
1 + \frac{R_0}{\rho(y)}\ \leq\ C\, \left(1+\frac{R_0}{\rho(x_0)}\right)^{N_0}.
\end{equation}
\item There exists $C>0$ such that
\begin{equation}\label{R0xy}
1+\frac{r}{\rho(y)}\ \leq \ C\, \left(1+\frac{R_0}{\rho(x_0)}\right)^\ga \left(1+\frac{r}{\rho(x)}\right),
\end{equation}
 for all $r>R_0$, where $\ga=N_0(1+\frac{N_0}{N_0+1})$.
\end{enumerate}
\end{coro}
\begin{proof}
Inequality \eqref{R0y} is a straightforward consequence of the left hand side of \eqref{rox_vs_roy}. Inequality \eqref{R0xy} follows from the right hand side of \eqref{rox_vs_roy} and then \eqref{R0y}.

\end{proof}

\begin{propo}[See \cite{DZ-HSH-99}]\label{propo:bol_critic}
There exists a sequence of points $x_j$, $j \ge 1$, in $\RR^d$,
so that the family $Q_j=B(x_j,\rho(x_j))$, $j\ge 1$, satisfies
\begin{enumerate}[(i)]
\item $\displaystyle \cup_j Q_j = \RR^d$.
\item \label{overlap} For every $\sigma\ge 1$ there exist constants $C$ and $N_1$ such that, $\sum_j \chi_{\sigma Q_j} \leq C\sigma^{N_1}$.
\end{enumerate}
\end{propo}

\begin{lem}\label{lem:casilema}
Let $V\in RH_{q}$ with $q>d/2$ and $\eps > \frac{d}{q}$. Then for any constant $C_1$ there exists a constant $C_2$ such that 
\begin{equation*}
\int_{B(x,C_1r)} \frac{V(u)}{|u-x|^{d-\eps}}du  \leq C_2\, r^{\eps-2} \left(\frac{r}{\rho(x)}\right)^{2-d/q},
\end{equation*}
if $0<r\le\rho(x)$.
\end{lem}

\section{The space  $BMO_\infty(\rho)$}\label{sec:bmo}

From the definition \eqref{defBMOtheta} given in the introduction, it is clear that  $BMO\subset BMO_\theta(\rho)\subset BMO_{\theta'}(\rho)$ for $0<\theta \leq \theta'$, and hence $BMO\subset BMO_\infty(\rho)$. Moreover, it is in general a larger class. For instance, when $\rho$ is constant (which corresponds to $V$ a positive constant) the functions $b_j(x)=|x_j|$, $1\leq j \leq d$, belong to $BMO_\infty(\rho)$ but not to $BMO$. Also, when $V(x)=|x|^2$ and $\LL$ becomes the Hermite operator, we obtain $\rho(x)\simeq\frac{1}{1+|x|}$ and we may take $b(x)=|x_j|^2$.

Given a Young function $\varphi$ and a locally integrable $f$ we consider the $\varphi$-average over a ball or a cube (denoted by $Q$) defined as
\begin{equation}\label{normaOrlicz}
\|f\|_{\varphi,Q}=\inf\left\{\la>0:\ \frac{1}{|Q|}\int_Q \varphi\left(\frac{|f|}{\la}\right) \leq 1\right\}.
\end{equation}

If we denote by $\tilde{\varphi}$ the conjugate Young function of $\varphi$, it is well known that the following version of H\"older inequality holds
\begin{equation}\label{Holderfi}
\frac{1}{|Q|}\int_Q |fg|\ \leq\ 2\,\|f\|_{\varphi,Q}\|g\|_{\tilde{\varphi},Q}. 
\end{equation}

Let us remind that for a function $b\in BMO(Q)$, as a consequence of the John-Nirenberg inequality (see for example \cite{MR1800316} p.151), we have
\begin{equation}\label{eqnorm}
\|b\|_{BMO(Q)} \simeq \sup_{B\subset Q} \|b-b_B\|_{\varphi,B},
\end{equation}
for certain Young functions $\varphi$. For instance $\varphi(t)=t^s$, $1<s<\infty$, or $\varphi(t)=e^t-1$.

For the spaces $BMO_\infty(\rho)$, we have a weaker version of this fact that will be enough to our purposes.

\begin{lem}\label{lem:JN}
Let $b\in BMO_\theta(\rho)$ and $\varphi$ such that \eqref{eqnorm} holds. Then there exist constants $C$ and $\theta'$ such that for every ball $B=B(x,r)$ we have
\begin{equation*}
\|b-b_{2^kB}\|_{\varphi,B} \ \leq \ C\,k\, [b]_\theta \left(1+\frac{2^kr}{\rho(x)}\right)^{\theta'}.
\end{equation*}
\end{lem}
\begin{proof}
For $k=1$ the proof follows the same lines than that of Proposition~3 in \cite{BHS_commut}. The case $k>1$ is a consequence of the case $k=1$ and the inequality
\begin{equation*}
\|b-b_{2^kB}\|_{\varphi,B}\ \leq \ \|b-b_{B}\|_{\varphi,B} + \frac{1}{\varphi^{-1}(1)}\sum_{i=1}^k |b_{2^iB}-b_{2^{i-1}B}|.
\end{equation*}

\end{proof}

\section{Weights}\label{sec:pesos}

As in \cite{BHS_schr-w}, we need classes of weights that are given in terms of the critical radius function \eqref{laro}. Given $p> 1$, we define $A_p^{\rho,\infty}=\cup_{\theta\ge 0} A_p^{\rho,\theta}$, where $A_p^{\rho,\theta}$ is the set of weights $w$ such that
\begin{equation*}
\left(\frac{1}{|B|}\int_B w\right)^{1/p} \left(\frac{1}{|B|}\int_{B} w^{-\frac{1}{p-1}}\right)^{1/p'}\ \leq\ C \left(1+\frac{r}{\rho(x)}\right)^\theta,
\end{equation*}
for every ball $B=B(x,r)$.

For $p=1$ we define $A_1^{\rho,\infty}=\cup_{\theta\ge 0} A_1^{\rho,\theta}$, where $A_1^{\rho,\theta}$ is the set of weights $w$ such that
\begin{equation}\label{Aptheta}
\frac{1}{|B|}\int_B w\ \leq\ C \ \left(1+\frac{r}{\rho(x)}\right)^{\theta}\inf_B w,
\end{equation}
for every ball $B=B(x,r)$. 

\begin{rem}\label{rem:Apcubos}
It is not difficult to see that in \eqref{Aptheta} it is equivalent to consider  cubes instead of balls, due to Proposition~\ref{propo:rho_xy}.
\end{rem}

These classes of weights, that contain Muckenhoupt weights, were introduced in \cite{BHS_schr-w}, where the next property is proven.

\begin{propo}\label{propo:Apmenos}
If $w\in A_p^{\rho,\infty}$, $1<p<\infty$, then there exists $\eps>0$ such that $w\in A_{p-\eps}^{\rho,\infty}$.
\end{propo}

The following results are extensions of very well known properties of $A_1$ weights.

\begin{lem}\label{lem:p1}
If $u\in A_1^{\rho,\infty}$, then there exists $\nu>1$ such that $u^\nu\in A_1^{\rho,\infty}$.
\end{lem}
\begin{proof}
This result follows immediately from the reverse H\"older type inequality valid for $A_p^{\rho,\infty}$ weights (see Lemma~5 in \cite{BHS_schr-w}).   
\end{proof}

For $\theta>0$ let us introduce the maximal function $M^\theta$ by
\begin{equation*}
M^\theta f (x) = \sup_{r>0} \frac{1}{\left(1+\frac{r}{\rho(x)}\right)^\theta} \frac{1}{|B(x,r)|} \int_{B(x,r)} |f|.
\end{equation*}

\begin{rem}\label{rem:delA1}
Observe that a weight $u$ belongs to $A_1^{\rho,\infty}$ if and only if there exists $\theta>0$ such that $M^\theta u \lesssim u$.
\end{rem}

\begin{lem}\label{lem:MA1}
Let $g\in L^1_{\text{loc}}$, $\theta\ge 0$ and $0<\de<1$, then $(M^\theta g)^\de \in A_1^{\rho,\infty}$.
\end{lem}
\begin{proof}
It is enough to prove that there exists $\be\ge 0$ such that for every ball $B_0=B(x_0,R_0)$,
\begin{equation}\label{infMde}
\frac{1}{|B_0|} \int_{B_0} (M^\theta g)^\de\ \lesssim\ \left(1+\frac{R_0}{\rho(x_0)}\right)^\be \inf_{B_0} (M^\theta g)^\de.
\end{equation}

We split $g=g_1 + g_2$, with $g_1=g\chi_{2B_0}$.

For $g_1$ we use the weak type $(1,1)$ of $M^\theta$ and Kolmogorov inequality to get for any $x\in B_0$,
\begin{equation*}
\frac{1}{|B_0|} \int_{B_0} (M^\theta g_1)^\de\ \lesssim\ \left(\frac{1}{|B_0|} \int_{2B_0} |g|\right)^\de\ \lesssim\ \left(1+\frac{R_0}{\rho(x)}\right)^{\theta\de} (M^\theta g(x))^\de.
\end{equation*}

Using \eqref{R0y} we arrive to the right hand side of \eqref{infMde}.

For the term with $g_2$ we have that for any $x$ and $y$ in $B(x_0,R_0)$
\begin{equation}\label{coso}
M^\theta g_2(x) \ \lesssim\ \left(1+\frac{R_0}{\rho(x_0)}\right)^{\ga\theta} M^\theta g_2(y),
\end{equation}
where $\ga$ is the constant appearing in \eqref{R0xy}.

In fact, considering a ball $B(x,r)$ with $r\ge R_0$ (otherwise the average of $g_2$ is zero), and using \eqref{R0xy} it follows
\begin{equation*}
\begin{split}
\frac{1}{\left(1+\frac{r}{\rho(x)}\right)^{\theta}} & \frac{1}{|B(x,r)|} \int_{B(x,r)} |g_2|\\
& \lesssim\ \left(1+\frac{R_0}{\rho(x_0)}\right)^{\ga\theta} \frac{1}{\left(1+\frac{r}{\rho(y)}\right)^{\theta}} \frac{1}{|B(y,Cr)|} \int_{B(y,Cr)} |g_2|
\end{split}
\end{equation*}
for any $y\in B_0$, leading to \eqref{coso}.

Raising \eqref{coso} to the $\de$ power and taking averages over $B_0$ respect to $x$ we arrive to the right hand side of \eqref{infMde} with $\be=\ga\theta\de$.

Finally, collecting the estimates for $g_1$ and $g_2$ the proof of the Lemma is finished.

\end{proof}

\section{Estimates of the Kernels}\label{sec:estK}

The operators $\rie$ and $\rie^*$ have singular kernels with values in $\RR^d$ that will be denoted by $\ke$ and $\ke^*$ respectively. For such kernels, we have the following estimates that are basically proved in \cite{shen} and \cite{chinos} (see also Lemma~3 in \cite{BHS_commut}).

\begin{lem}\label{lem:chinos}
Let $V\in RH_{q}$ with $q>d/2$.
\begin{enumerate}[(i)]
\item For every $N$ there exists a constant $C_N$ such that
\begin{equation}\label{kes}
|\ke^*(x,y)|\ \leq\ \frac{C_N\left(1+\frac{|x-y|}{\rho(x)}\right)^{-N}}{|x-y|^{d-1}} \left(\int_{B(y,|x-y|/4)} \frac{V(u)}{|u-y|^{d-1}}du + \frac{1}{|x-y|}\right).
\end{equation}
Moreover, the last inequality also holds with $\rho(x)$ replaced by $\rho(y)$.

\item For every $N$ and $0<\de<\min\{1,2-d/q\}$ there exists a constant $C$ such that 
\begin{equation}\label{kesdif}
\begin{split}|\ke^*(x,z)&-\ke^*(y,z)| \ \leq \\ & \frac{C\,|x-y|^{\de}\left(1+\frac{|x-z|}{\rho(x)}\right)^{-N}}{|x-z|^{d-1+\de}} \left(\int_{B(z,|x-z|/4)} \frac{V(u)}{|u-z|^{d-1}}du + \frac{1}{|x-z|}\right)
\end{split}
\end{equation}
whenever $|x-y|<\frac{2}{3}|x-z|$. Moreover, the last inequality also holds with $\rho(x)$ replaced by $\rho(z)$.

\item If $\mathbf{K}^*$ denotes the $\RR^d$ vector valued kernel of the adjoint of the classical Riesz operator, then
\begin{equation}\label{kesyclas}
\begin{split}
|\ke^*(x,z)&-\mathbf{K}^*(x,z)| \leq \\ & \frac{C}{|x-z|^{d-1}}  \left(\int_{B(z,|x-z|/4)} \frac{V(u)}{|u-z|^{d-1}}du + \frac{1}{|x-z|}\left(\frac{|x-z|}{\rho(x)}\right)^{2-\frac{d}{q}}\right),
\end{split}
\end{equation}
whenever $|x-z|\leq \rho(x)$.
\item \label{qgrande} When $q> d$, the term involving $V$ can be dropped from inequalities \eqref{kes} and \eqref{kesyclas}.

\item  If $q> d$, the term involving $V$ can be dropped from inequalities \eqref{kes}, \eqref{kesdif} and \eqref{kesyclas}. 
\end{enumerate}
\end{lem}

The following lemma improves a result appearing in \cite{chinos}.

\begin{lem}\label{lem:hormander}
Let $V\in RH_q$ with $d/2<q<d$ and $s$ such that $\frac{1}{s}=\frac{1}{q} -\frac{1}{d}$. Then the kernel $\ke$ satisfies the following H\"ormander type inequality
\begin{equation}\label{hormander}
\sum_k k(2^k r)^{d/s'} \left(1 + \frac{2^kr}{\rho(x_0)}\right)^\theta \left(\int_{|x-x_0|\sim 2^k r} |\ke(x,y)-\ke(x,x_0)|^s\,dx\right)^{1/s} \ \leq \ C_\theta,
\end{equation}
whenever $|y-x_0|<r$, and $r\ge 0$.
\end{lem}
\begin{proof}
We follow the lines of the proof of Lemma~4 in \cite{chinos} but performing a more careful estimate.

Using \eqref{kesdif} we get
\begin{equation*}
\begin{split}
& \left(\int_{|x-x_0|\sim 2^k r} |\ke(x,y)-\ke(x,x_0)|^s\,dx \right)^{1/s} \\
& \hspace{30pt} \lesssim \ (2^kr)^{(1-d)} 2^{-k\de} \left(1 + \frac{2^kr}{\rho(x_0)}\right)^{-N}\ \left(\|I_1(V\chi_{B(x_0,2^kr)})\|_s + (2^kr)^{\frac{d}{s}-1}\right),
\end{split}
\end{equation*}
where $I_1$ stands for the fractional integral operator of order one.

The estimate of \eqref{hormander} involving the second term above follows easily.

Now, from the boundedness of $I_1$ and the fact that $V\in RH_q$,
\begin{equation}\label{laI1}
\|I_1(V\chi_{B(x_0,2^kr)})\|_s \ \leq \ (2^kr)^{-\frac{d}{q'}} \int_{B(x_0,2^kr)} V,
\end{equation}
where the last integral can be estimated as
\begin{equation}\label{paralaV}
\int_{B(x_0,2^kr)} V \leq \ (2^kr)^{d-2}\left(\frac{2^kr}{\rho(x_0)}\right)^{\be}
\end{equation}
with $\be=2-\frac{d}{q}$ when $2^kr\leq\rho(x_0)$ and $\be=\mu d$, $\mu\ge 1$ in other case (see \cite{BHS_schr-w}).

Therefore we can bound the left hand side of \eqref{laI1} by either $\rho(x_0)^{\frac{d}{q}-2}$ or
\begin{equation*}
(2^k r)^{\frac{d}{q}-2} \left(\frac{2^k r}{\rho(x_0)}\right)^{\mu d} \ \text{ with $\mu\ge 1$}.
\end{equation*}
Now, to finish the estimate of the sum on the left hand side of \eqref{hormander} we first sum over $k\in J_1 = \{k\in\NN:\ 2^kr\leq \rho(x_0)\}$. For such sum, using the above estimates and that $2-\frac{d}{q}>0$ we get the bound
\begin{equation*}
\sum_k k\, 2^{-k\de}\,(2^k r)^{1-d+\frac{d}{s'}} \rho(x_0)^{\frac{d}{q}-2} \ \lesssim \ \sum_k k\, 2^{-k\de} \ \lesssim 1.
\end{equation*}
Similarly, the other sum can be bounded by
\begin{equation*}
\sum_k k\, 2^{-k\de}\,(2^k r)^{-1-d+\frac{d}{q}+\frac{d}{s'}} \left(\frac{2^kr}{\rho(x_0)}\right)^{-N+\theta+\mu d} \ \lesssim \ \sum_k k\, 2^{-k\de} \ \lesssim 1,
\end{equation*}
choosing $N$ large enough. 

\end{proof}

\begin{lem}\label{lem:kcor}
Let $V\in RH_q$, $\frac{d}{2}<q<d$, and $\frac{1}{s}=\frac{1}{q}-\frac{1}{d}$. Then, for all $N$ there exists $C_N$ such that for any ball $B=B(z,r)$ with $r\ge\rho(z)$, $y\in B$ and $B^k=2^kB$, the inequality
\begin{equation}\label{estkcor}
\left(\int_{B^k\setminus B^{k-1}}|\ke(x,y)|^s\, dx\right)^{1/s} \ \leq \ C_N (2^kr)^{-1-\frac{d}{q'}}\left(\frac{\rho(z)}{2^kr}\right)^{N-\mu d}
\end{equation}
holds for some $\mu\ge 1$, which depends only on the constants appearing in the doubling condition that $V$ satisfies.
\end{lem}
\begin{proof}
From Lemma~\ref{lem:chinos} we know
\begin{equation}\label{kesder}
|\ke(x,y)|\ \leq\ \frac{C_N\left(1+\frac{|x-y|}{\rho(x)}\right)^{-N}}{|x-y|^{d-1}} \left(\int_{B(y,2|x-y|)} \frac{V(u)}{|u-y|^{d-1}}du + \frac{1}{|x-y|}\right).
\end{equation}
Now, for $B$ and $y$ as in the statement and $x\in B^k\setminus B^{k-1}$ we have $B(y,2|x-y|)\subset B^{k+1}$. Also, since $x\in B^{k+1}$ we may use Corollary~\ref{coro:dero} to reduce \eqref{kesder}, with perhaps a different $N$, to
\begin{equation*}
|\ke(x,y)| \ \leq \ C_N (2^kr)^{1-d}\left(\frac{\rho(z)}{2^kr}\right)^{N} \left(\frac{1}{2^kr} + I_1(\chi_{B^{k+1}}V)(y)\right).
\end{equation*}
Therefore,
\begin{equation*}
\left(\int_{B^k\setminus B^{k-1}}|\ke(x,y)|^s\, dx\right)^{1/s} \ \lesssim \ (2^kr)^{1-d}\left(\frac{\rho(z)}{2^kr}\right)^{N}\left( (2^kr)^{\frac{d}{s}-1} + \|I_1(\chi_{B^{k+1}}V)\|_s\right).
\end{equation*}

According to inequalities \eqref{laI1} and \eqref{paralaV} we have
\begin{equation*}
\|I_1(\chi_{B^{k+1}}V)\|_s \lesssim \ (2^kr)^{-\frac{d}{q'}+d-2}\left(\frac{2^kr}{\rho(z)}\right)^{\mu d}
\end{equation*}
for some $\mu\ge 1$. Thus, plugging this estimate and using that $r\ge\rho(z)$ and that $\frac{d}{s}-1 = \frac{d}{q}-2= -\frac{d}{q'}+d-2$, we arrive to \eqref{estkcor}.

\end{proof}

\section{$L^p$ inequalities}\label{sec:lpineq}

For an operator $T$ we associate the {\it local} and {\it global} operators of $T$ as
\begin{equation*}
T_{\text{loc}}f(x)= T(f\chi_{B(x,\rho(x))})(x)
\end{equation*}
and
\begin{equation*}
T_{\text{glob}}f(x)= T(f\chi_{B(x,\rho(x))^c})(x)
\end{equation*}
respectively, where the first integral should be understood, if necessary, in the sense of principal value.

In the following Theorem, we use a larger classes of weights $A_p^{\rho,\text{loc}}$ already defined in \cite{BHS_schr-w} as those weights that satisfy the classical $A_p$ condition for balls $B(x,r)$ with $r\leq \rho(x)$. From the well known proof for $A_p$ classes it is easy to derive the following result.

\begin{propo}\label{propo:Apmenosloc}
If $w\in A_p^{\rho,\text{loc}}$, $1<p<\infty$, then there exists $\eps>0$ such that $w\in A_{p-\eps}^{\rho,\text{loc}}$.
\end{propo}

Let us observe that a function $b\in BMO_\infty(\rho)$ has bounded mean oscillations over sub-critical balls, that is balls $B(x,r)$ with $r\leq \rho(x)$. For the next result we shall denote $BMO_\text{loc}(\rho)$ the set of functions with the latter property.

\begin{rem}\label{rem:bmoloc}
Notice that using Proposition~\ref{propo:bol_critic} it is possible to prove that for each constant $C$ we have $BMO_\text{loc}(\rho) = BMO_\text{loc}(C\rho)$ and the norms are equivalent with a constant that depends on $C$.
\end{rem}

\begin{teo}\label{teo:clasiloc}
Let $\rho$ a function satisfying \eqref{rox_vs_roy} and $b\in BMO_\text{loc}(\rho)$, then $(R_b)_{\text{loc}}$ are bounded on $L^p(w)$, $1< p<\infty$, for $w\in A_p^{\rho,\text{loc}}$.
\end{teo}
\begin{proof}
Let $\{Q_j\}_j$ be a covering by critical balls as in Proposition~\ref{propo:bol_critic}. It is possible to find a constant $\be$ such that if $\Qt_j=\be Q_j$ then $\cup_{x\in Q_j}B(x,\rho(x)) \subset \Qt_j$.

From Lemma~1 in \cite{BHS_schr-w} a weight in $A_p^{\rho,\text{loc}}$ when restricted to some $\Qt_j$ can be extended to $\RR^d$ as an $A_p$ weight preserving the $A_p^{\rho,\text{loc}}$ constant. This kind of result can be extended also to $BMO$ functions because of their well known relationship with $A_p$ weights \cite{MR1800316}. Therefore, given $b\in BMO_{\text{loc}}(\rho)$ and any $\Qt_j$, there is an extension of $b\chi_{\Qt_j}$ to the whole $\RR^d$ that we call $b_j$ belonging to $BMO$ and with norm bounded by $[b]_{\text{loc}}$, the natural norm in $BMO_{\text{loc}}(\rho)$.

For $x\in Q_j$, since $b=b_j$ on $\Qt_j$, we have
\begin{equation*}
|(R_b)_{\text{loc}} f(x)|\ \leq\  |(R_b)_{\text{loc}} f(x) - (R_b)(\chi_{\Qt_j}f)(x)| + |(R_{b_j})(\chi_{\Qt_j}f)(x)|.
\end{equation*}

The first term can be bounded as
\begin{equation}\label{paralaR}
\begin{split}
|(R_b)_{\text{loc}} f(x) - (R_b)(\chi_{\Qt_j}f)(x)|\ & \lesssim\ \int_{\Qt_j\setminus B(x,\rho(x))} \frac{|f(y)|\,|b_j(x)-b_j(y)|}{|x-y|^{d}}\, dy\\
& \lesssim\ [b]_{\text{loc}}\, M_{s,\text{loc}}(f)(x),
\end{split}
\end{equation}
for each $s>1$, with $M_{s,\text{loc}}f(x)=\sup_{B} \left(\frac{1}{|B|}\int_B |f|^s\right)^{1/s}$ where the sup is taken over sub-critical balls respect to the function $\be\rho$. Notice that to get the last inequality we made use of Remark~\ref{rem:bmoloc}. Then, since $w=w_j$ on $\Qt_j$,

\begin{equation*}
\begin{split}
\int_{\RR^d}|(R_b)_{\text{loc}} f|^pw\ \leq\  \sum_j \left( [b]_{\text{loc}} \int_{Q_j} |M_{s,\text{loc}}(f)|^p\,w + \int_{Q_j} |(R_{b_j})(\chi_{\Qt_j}f)|^p\,w_j\right).
\end{split}
\end{equation*}

By Proposition~\ref{propo:Apmenosloc} and the boundedness of $M_{1,\text{loc}}$ with $A_p^{\rho,\text{loc}}$ weights (see Theorem~1 in \cite{BHS_schr-w}), we obtain the desired estimate for the first term.

For the second term we use that, since $b_j$ belongs to $BMO$, the commutator $R_{b_j}$ is bounded on $L^p(w_j)$ where $w_j$ is an $A_p$ extension of $w\chi_{\Qt_j}$ to all $\RR^d$ with $A_p$ constant depending only on the $A_p^{\rho,\text{loc}}$ constant of the original weight $w$. We also notice that the operator norm of $R_{b_j}$ is independent of $j$.

\end{proof}

Now, give a technical Lemma that we will need in the proof of Theorem~\ref{teo:ow}.

\begin{lem}\label{lem:laclaim}
Let $\rho$ a function satisfying \eqref{rox_vs_roy} and $b\in BMO_\theta(\rho)$. Let $w$ verifying the reverse H\"older's inequality \eqref{reverse} for $q=\delta$ and $B$ any sub-critical ball. Then, given any $p,\nu >0$, there exists a constant $M>0$ such that
\begin{equation}\label{claim}
\begin{split}
\int_{B} w(x) \left(\int_{\la B}|b(x)-b(y)|^{\nu}\,dy\right)^{p/\nu}dx \ \lesssim\ \la^{M} [b]_{\theta}^p\, |B|^{p/\nu} w(B)
\end{split}
\end{equation}
for any sub-critical ball $B$ and all $\la\ge 1$.
\end{lem}
\begin{proof}
Let $B=B(x_0,r)$ with $r\leq \rho(x_0)$. The left side of \eqref{claim} can be bounded by
\begin{equation}\label{2terms}
\begin{split}
\la^{dp/\nu}|B|^{p/\nu}\int_{B} w(x)\, |b(x)-b_{\la B}|^p\,dx\  +\  w(B) \left(\int_{\la B}|b(y)-b_{\la B}|^{\nu}\,dy\right)^{p/\nu}
\end{split}
\end{equation}
For the first term of the last expression we use H\"older's inequality with exponent $\delta$ and the assumption on $w$ and Lemma~\ref{lem:JN} to bound it by
\begin{equation*}
\begin{split}
\la^{dp/\nu}|B|^{p/\nu-1/\de'} w(B) & \left(\int_{\la B}|b(x)-b_{\la B}|^{p\de'}\,dx\right)^{1/\de'} \\ 
& \lesssim \ \la^{dp/\nu+d/\de'}|B|^{p/\nu} w(B)\, [b]_{\theta}^p \left(1 + \frac{r\la}{\rho(x_0)}\right)^{p\theta'}.
\end{split}
\end{equation*}
Using that $r\leq \rho(x_0)$ and  $\la\ge 1$, we arrive to the desired estimate with $M=d(p/\nu+1/\de'+\theta'p/d)$.

For the second term of \eqref{2terms} we use again Lemma~\ref{lem:JN} to get the bound 
\begin{equation*}
\la^{p/\nu} w(B)\, [b]_{\theta}^p\, \left(1 + \frac{r\la}{\rho(x_0)}\right)^{p\theta'} |B|^{p/\nu},
\end{equation*}
and the proof is finished proceeding as before.
\end{proof}

\begin{teo}\label{teo:Lp}
Let $V\in RH_q$ and $b\in BMO_\infty(\rho)$.
\begin{enumerate}[(i)]
\item \label{Lp-qd} If $q\ge d$, the operators $\rie_b$ and $\rie^*_b$ are bounded on $L^p(w)$, $1<p<\infty$, for $\displaystyle w\in A_{p}^{\rho,\infty}$.
\item \label{Lp-qdsobre2} If $d/2< q <d$, and $s$ is such that $\frac{1}{s}=\frac{1}{q}-\frac{1}{d}$, the operator $\rie_b^*$ is bounded on $L^p(w)$, for $s'<p<\infty$ and $\displaystyle w\in A_{p/s'}^{\rho,\infty}$ and hence by duality $\rie_b$ is bounded on $L^p(w)$, for $1<p<s$, with $w$ satisfying $w^{-\frac{1}{p-1}}\in A_{p'/s'}^{\rho,\infty}$.
\end{enumerate}
\end{teo}
\begin{proof}
First of all, notice that there is no need to consider $q=d$ since in that case there exists an $\eps>0$ such that $V\in RH_{d+\eps}$.

We begin giving estimates for $\rie_b^*$.

Now we write\begin{equation}\label{sumaR}
(\rie_b^*)f = (R_b^*)_{\text{loc}}f\ +\ (\rie_b^*)_{\text{glob}}f\ +\ [(\rie_b^*)_{\text{loc}} - (R_b^*)_{\text{loc}}]f.
\end{equation}

As a consequence of Theorem~\ref{teo:clasiloc} the first term is bounded on $L^p(w)$ for $w\in A_{p}^{\rho,\text{loc}}$, $1<p<\infty$. Since $\displaystyle w\in A_{p}^{\rho,\infty}\subset A_{p}^{\rho,\text{loc}}$, $1\leq p<\infty$, and $\displaystyle w\in A_{p/s'}^{\rho,\infty}\subset A_{p}^{\rho,\text{loc}}$, $s'< p<\infty$, all the conclusions for $(R_b^*)_{\text{loc}}$ hold.

For the second term of \eqref{sumaR} we use \eqref{kes} to obtain
\begin{equation*}
 |(\rie^*_b)_{\text{glob}}f(x)|\ \leq\ \int_{B(x,\rho(x))^c} |b(y)-b(x)|\,|\ke^*(x,y)||f(y)|\,dy\ \lesssim\ g_1(x) + g_2(x),
\end{equation*}
with
\begin{equation*}
\begin{split}
 g_1(x) & = \sum_{k=0}^\infty 2^{-kN} g_{1,k}(x),
\end{split}
\end{equation*}
where $g_{1,k}(x)=\frac{1}{(2^k\rho(x))^{d}} \int_{B(x,2^k\rho(x))} |b(y)-b(x)|\,|f(y)|\,dy$, and
\begin{equation*}
 g_2(x) = \sum_{k=0}^\infty 2^{-kN} g_{2,k}(x).
\end{equation*}

\begin{equation*}
g_{2,k}(x)=\frac{1}{(2^k\rho(x))^{d-1}} \int_{B(x,2^k\rho(x))} \left(\int_{B(x,2^k\rho(x))} \frac{V(u)}{|u-y|^{d-1}}du\right) |b(y)-b(x)|\,|f(y)|\,dy
\end{equation*}

To deal with $g_1$, let $\sigma=c_0 2^{\frac{N_0}{N_0+1}}$, with $N_0$ and $c_0$ as in Proposition~\ref{propo:rho_xy}. Let $\{Q_j\}$ be the family given by Proposition~\ref{propo:bol_critic} and set $\Qt_j=\sigma Q_j$. Clearly, we have
\begin{equation}\label{contencion}
\cup_{x\in Q_j} B(x,\rho(x)) \subset \Qt_j.
\end{equation}
Denoting $\Qt_j^k=2^k\Qt_j$, then $2^kB_x \subset \Qt_j^k$ and $\rho(x)\simeq \rho(x_j)$, whenever $x\in Q_j$. Therefore, by H\"older's inequality with $\ga$ and $\nu$ such that $\frac{1}{p}+\frac{1}{\ga}+\frac{1}{\nu}=1$,
\begin{equation*}
\begin{split}
\int_{Q_j} (g_{1,k})^pw\ & \lesssim\ \int_{Q_j} \left(\frac{1}{|\Qt_j^k|}\int_{\Qt_j^k} |b(x)-b(y)|\,|f(y)| \,dy \right)^p w(x)\,dx\\
& \lesssim\ \frac{1}{|\Qt_j^k|^p} \left(\int_{\Qt_j^k} w^{-\ga/p}\right)^{p/\ga}  \left(\int_{\Qt_j^k} |f|^p\,w\right)\ \times\\
& \hspace{80pt} \int_{Q_j} w(x) \left(\int_{\Qt_j^k}|b(x)-b(y)|^{\nu}\,dy\right)^{p/\nu}dx,\\
& \lesssim\ 2^{kM} [b]_\theta^p w(\Qt_j^k) |\Qt_j^k|^{\frac{p}{\nu}-p} \left(\int_{\Qt_j^k} w^{-\ga/p}\right)^{p/\ga}  \int_{\Qt_j^k} |f|^p\,w\\
\end{split}
\end{equation*}
for some $M>0$, where in the last inequality we have used Lemma~\ref{lem:laclaim} for $\theta$ such that $b\in BMO^\theta(\rho)$.

From Proposition~\ref{propo:Apmenos} we can choose $\ga$ close enough to $p'$ in such a way that $w\in A_{1+p/\ga}^{\rho,\eta}$ for some $\eta>0$. Therefore, for some $M_1>0$, we get
\begin{equation*}
\begin{split}
\int_{Q_j} (g_{1,k})^pw\ & \lesssim\ 2^{k M_1} [b]_\theta^p \int_{\Qt_j^k} |f|^p\,w
\end{split}
\end{equation*}
and hence for $M_1'>0$,
\begin{equation*}
\begin{split}
\|g_1\|_{L^p(w)} & \lesssim \sum_k 2^{-kN} \|g_{1,k}\|_{L^p(w)}\\
& \lesssim \sum_k 2^{-kN} \left(\sum_j \int_{Q_j}g_{1,k}^p\,w\right)^{1/p}\\
& \lesssim [b]_\theta \sum_k 2^{k(-N+M_1')} \left(\sum_j \int_{\Qt_j^k} |f|^p\,w \right)^{1/p}\\
& \lesssim [b]_\theta \|f\|_{L^p(w)} \sum_k 2^{k(-N+M_1'+N_1)},
\end{split}
\end{equation*}
where in the last inequality is due to Proposition~\ref{propo:bol_critic}. Choosing $N$ large enough the last series is convergent.

Regarding $g_2$, according to Lemma~\ref{lem:chinos}, we only have to consider $\frac{d}{2}<q<d$. 

Observe that for $x\in Q_j$ we have
\begin{equation*}
\begin{split}
\int_{B(x,2^k\rho(x))} \frac{V(u)}{|u-y|^{d-1}}du\ \lesssim\ I_1(\chi_{\Qt_j^k}V)(y),
\end{split}
\end{equation*}
where $I_1$ is the classical Fractional Integral of order 1.

Therefore, by H\"older's inequality with $\ga$ and $\nu$ such that $\frac{1}{p}+\frac{1}{s}+\frac{1}{\nu}+\frac{1}{\ga}=1$,
\begin{equation*}
\begin{split}
\int_{Q_j} (g_{2,k})^pw\ & \lesssim\ \int_{Q_j} \left(\frac{1}{|\Qt_j^k|^{1-1/d}}\int_{\Qt_j^k} |b(x)-b(y)|\,|f(y)|\, I_1(V\chi_{\Qt_k^j})(y)\,dy \right)^p w(x)\,dx\\
& \lesssim\ \frac{1}{|\Qt_j^k|^{p(1-1/d)}} \left(\int_{\Qt_j^k} w^{-\ga/p}\right)^{p/\ga}  \|\chi_{\Qt_j^k}f\|_{L^p(w)}^p \|I_1(\chi_{\Qt_j^k} V)\|_{s}^p\ \times\\
& \hspace{80pt} \int_{Q_j} w(x) \left(\int_{\Qt_j^k}|b(x)-b(y)|^{\nu}\,dy\right)^{p/\nu}dx.
\end{split}
\end{equation*}

Recall that $V\in RH_q$ for some $q>1$ implies that $V$ satisfies the doubling condition, i.e., there exist constants $\mu\ge 1$ and $C$ such that

\begin{equation*}
\int_{tB} V \leq C\,t^{d\mu}\int_{B} V,
\end{equation*}
holds for every ball $B$ and $t>1$. Therefore, due to the boundedness of $I_1$ from $L^{q}$ into $L^{s}$, and the assumptions on $V$,
\begin{equation*}
\begin{split}
\|I_1(\chi_{\Qt_j^k} V)\|_{s}\ & \lesssim\ \|\chi_{\Qt_j^k} V\|_{q} \lesssim\ |\Qt_j^k|^{-1/q'} \int_{\Qt_j^k} V\\
& \lesssim\ 2^{kd\mu} |\Qt_j^k|^{-1/q'} \int_{\Qt_j} V \lesssim\ 2^{kd(\mu-1+\frac{2}{d})} |\Qt_j^k|^{\frac{1}{q}-\frac{2}{d}}
\end{split}
\end{equation*}
where the last inequality follows from the definition of $\rho$ (see \eqref{laro}). With this estimate and using the claim, we proceed as in the case of $g_1$, choosing this time $\ga$ such that $1+\frac{p}{\ga}$ is close enough to $\frac{p}{s'}$, to obtain
\begin{equation*}
\begin{split}
\int_{Q_j} (g_{2,k})^pw\ & \lesssim\ 2^{kM_2} \int_{\Qt_j^k} |f|^p\,w
\end{split}
\end{equation*}
for some $M_2$, leading to the desired estimate.

Now we have to deal with the term  $[(\rie^*_b)_{\text{loc}}- (R^*_b)_{\text{loc}}]f$ of \eqref{sumaR}. By using estimate \eqref{kesyclas}, we have
\begin{equation*}
|[(\rie^*_b)_{\text{loc}}- (R^*_b)_{\text{loc}}]f(x)| \ \lesssim\ h_1(x) + h_2(x)
\end{equation*}
where
\begin{equation*}
h_1(x)\ =\ \sum_k 2^{-k(2-d/q)}h_{1,k}(x),
\end{equation*}
with
\begin{equation*}
h_{1,k}(x)\ =\ 2^{kd}\rho(x)^{-d}\int_{B(x,2^{-k}\rho(x))} |f(y)|\,|b(x)-b(y)|\,dy
\end{equation*}
and
\begin{equation*}
h_2(x)\ \lesssim\ \sum_{k=0}^\infty 2^{k(d-1)} h_{2,k}(x),
\end{equation*}
where
\begin{equation*}
h_{2,k}(x) = \rho(x)^{-d+1}\!\int_{B(x,2^{-k}\rho(x))} |f(y)|\,|b(x)-b(y)|  \left(\int_{B(y,|x-y|/4)} \frac{V(u)}{|u-y|^{d-1}}du \right)\,dy.
\end{equation*}

Let us take a covering $\{Q_j\}$ as before. For each $j$ and $k$ there exist $2^{dk}$ balls of radio $2^{-k}\rho(x_j)$, $B_l^{j,k}=B(x_l^{j,k},2^{-k}\rho(x_j))$ such that $Q_j \subset \cup_{l=1}^{2^{dk}}B_l^{j,k} \subset 2Q_j$ and $\sum_{l=1}^{2^{dk}}\chi_{B_l^{j,k}}\leq 2^d$. Moreover, this construction can be done in a way that for each $k$ the family of a fixed dilation $\{\BB_l^{j,k}\}_{j,l}$ is a covering of $\RR^d$ such that
\begin{equation}\label{solapB}
\sum_j\sum_{l=1}^{2^{dk}}\chi_{\BB_l^{j,k}}\leq C,
\end{equation}
with the constant $C$ independent of $k$. To our purpose we take the dilation $\BB_l^{j,k}= 5 c_0 B_l^{j,k}$ (where $c_0$ appears in \eqref{rox_vs_roy}).

Observe that if $x\in B_l^{j,k}$, $B(x,2^{-k}\rho(x))\subset \Bt_l^{j,k}$ and  $\rho(x)\simeq \rho(x_j)$. Then
\begin{equation*}
\begin{split}
h_{1,k}(x)\ & \lesssim\ 2^{kd}\rho(x_j)^{-d} \int_{\Bt_l^{j,k}} |f(y)|\, |b(x)-b(y)|\,dy,
\end{split}
\end{equation*}

By H\"older's inequality with $\ga$ and $\nu$ as for $g_1$, and using Lemma~\ref{lem:laclaim} and Proposition~\ref{propo:Apmenos}, we have
\begin{equation*}
\begin{split}
\int_{B_l^{j,k}} (h_{1,k})^pw\ & \lesssim\ \int_{B_l^{j,k}} \left(2^{kd}\rho(x_j)^{-d} \int_{\Bt_l^{j,k}} |b(x)-b(y)|\,|f(y)| \,dy \right)^p w(x)\,dx\\
& \lesssim\ 2^{kdp}\rho(x_j)^{-dp} \left(\int_{\Bt_l^{j,k}} w^{-\ga/p}\right)^{p/\ga}  \left(\int_{\Bt_l^{j,k}} |f|^p\,w\right)\ \times\\
& \hspace{80pt} \int_{B_l^{j,k}} w(x) \left(\int_{\Bt_l^{j,k}}|b(x)-b(y)|^{\nu}\,dy\right)^{p/\nu}dx\\
& \lesssim\ [b]_\theta^p \int_{\Bt_l^{j,k}} |f|^p\,w.
\end{split}
\end{equation*}

Adding over $j$ and $l$, and using the bounded overlapping property \eqref{solapB},
\begin{equation*}
\|h_{1,k}\|_{L^p(w)}\ \lesssim\ [b]_\theta\, \|f\|_{L^p(w)},
\end{equation*}
and thus we obtain the desired estimate for $h_1$.

To deal with $h_2$, we use that $I_1$ is bounded from $L^q$ into $L^s$, together with Lemma~\ref{lem:casilema}, to get
\begin{equation}\label{I1V}
\begin{split}
\|I_1(\chi_{\Bt_l^{j,k}} V)\|_{s}\ & \lesssim \|\chi_{\Bt_l^{j,k}} V\|_{q}\\
& \lesssim\ |\Bt_l^{j,k}|^{-1+1/q} \int_{\Bt_l^{j,k}} V\\
& \lesssim\ \rho(x_j)^{-2+d/q}.
\end{split}
\end{equation}

Now, we proceed as for $h_1$ but this time we apply H\"older's inequality with $\ga$ and $\nu$ such that $\frac{1}{p}+\frac{1}{s}+\frac{1}{\nu}+\frac{1}{\ga}=1$,
\begin{equation*}
\begin{split}
\int_{B_l^{j,k}} (h_{2,k})^pw\ 
& \lesssim\ \frac{1}{\rho(x_j)^{p(d-1)}} \left(\int_{\Bt_l^{j,k}} w^{-\ga/p}\right)^{p/\ga}  \|\chi_{\Bt_l^{j,k}}f\|_{L^p(w)}^p \|I_1(\chi_{\Bt_l^{j,k}} V)\|_{s}^p\ \times\\
& \hspace{80pt} \int_{B_l^{j,k}} w(x) \left(\int_{\Bt_l^{j,k}}|b(x)-b(y)|^{\nu}\,dy\right)^{p/\nu}dx\\&
 \lesssim\ [b]_\theta^p\, 2^{-kpd(1-\frac{1}{q}+\frac{1}{d})} \|\chi_{\Bt_l^{j,k}}f\|_{L^p(w)}^p
\end{split}
\end{equation*}
Therefore, with the same argument as for $h_1$, and adding over $k$,
\begin{equation*}
\|h_{2,k}\|_{L^p(w)}\ \lesssim\ [b]_\theta\, \|f\|_{L^p(w)}.
\end{equation*}
and we finish the proof of the theorem.
\end{proof}

\section{An Orlicz weak estimate for the case $p=1$}\label{sec:debil}

In the next lemma we will use the notation $P(x,r)$ to denote the cube of center $x$ and side $2r$.

\begin{lem}\label{lem:CZdecom}
Let $\rho$ be a function satisfying \eqref{rox_vs_roy} and $\theta\ge 0$ fixed. Then for any $\la>0$ there exists an at most countable family of cubes $\{P_j\}$, $P_j=P(x_j,r_j)$ such that
\begin{equation}\label{CZcub}
\left(1 + \frac{r}{\rho(x_j)}\right)^{\theta} \la \ \leq \ \frac{1}{|P_j|}\int_{P_j}|f| \ \leq\ C\, \la\left(1 + \frac{r}{\rho(x_j)}\right)^{\sigma},
\end{equation}
for some $\sigma\ge \theta$, depending only on the constants appearing in \eqref{rox_vs_roy}, and
\item 
\begin{equation}\label{puntual}
|f(x)|\leq \la,\ \ \text{a.e.}\ \ \ x\notin\cup_{j}P_j.
\end{equation}
\end{lem}
\begin{proof}
First, let us observe that for any cube $P=P(x,r)$,
\begin{equation*}
\left(1 + \frac{r}{\rho(x)}\right)^{-\theta} \frac{1}{|Q|}\int_{Q}|f| \ \leq \ \frac{1}{|Q|}\int_{\RR^d}|f|,
\end{equation*}
and the right hand side tends to zero when $r$ goes to infinity. Therefore we may start the Calderon-Zygmund decomposition process with some $r_0$-grid such that
\begin{equation}\label{menorqla}
\left(1 + \frac{r_0}{\rho(z)}\right)^{-\theta} \frac{1}{|P(z,r_0)|}\int_{P(z,r_0)}|f| \ \leq \ \la,
\end{equation}
for any cube in the grid. We divide dyadically the cubes selecting those for which the average on the left turns greater than $\la$.

Continuing dividing those cubes that have not been selected we obtain a sequence of $P_j$ satisfying the left inequality of \eqref{CZcub}.

To check the other inequality, observe that if $P_j=P(x_j,r_j)$ was selected, then $P_j$ is contained in a cube $P(y,2r_j)$ satisfying \eqref{menorqla} for some $y$. Hence
\begin{equation*}
\frac{1}{|P_j|}\int_{P_j}|f| \ \lesssim \ \left(1 + \frac{2r}{\rho(y)}\right)^{\theta}\ \lesssim\ \left(1 + \frac{r_j}{\rho(x_j)}\right)^{N_0\theta},
\end{equation*}
where in the last inequality we used \eqref{rox_vs_roy}.

Next, if $x\notin \cup_j P_
j$ there exists a sequence of cubes containing $x$ and with radius tending to zero satisfying \eqref{menorqla}. Since $\rho$ is continuous and positive \eqref{puntual} follows from the Lebesgue's differentiation theorem.

\end{proof}

\begin{teo}\label{teo:ow}
Let $V\in RH_q$  and $b\in BMO_\infty(\rho)$.
\begin{enumerate}[(i)]
\item \label{weak-qd} If $q\ge d$ and $w\in A_1^{\rho,\infty}$, then there exists a constant $C$ such that for every $f\in \Lloc$ and $\la>0$,
\begin{equation}\label{tipod}
w(\{|\rie_bf|>\la\}) \leq C \int_{\RR^d} \frac{|f|}{\la}\left(1+\log^{+} \frac{|f|}{\la}\right) w.
\end{equation}

\item \label{weak-qdsobre2} If $d/2< q <d$ and $w^{s'}\in A_1^{\rho,\infty}$, with $\frac{1}{s}=\frac{1}{q}-\frac{1}{d}$, inequality \eqref{tipod} holds.
\end{enumerate}
\end{teo}
\begin{proof}
First, let us observe that \eqref{weak-qd} can be deduced from \eqref{weak-qdsobre2}. In fact, if $w\in A_1^{\rho,\infty}$ there exist $\ga_0>1$ such that $w^\ga\in A_1^{\rho,\infty}$, for $1\leq\ga\leq\ga_0$, according to Lemma~5. Oh the other hand if $V\in RH_q$ for $q\ge d$ it certainly belongs to $RH_s$ for any $s<d$. In particular we may choose $\frac{d}{2}<s<d$ and $\ga>\ga_0$ such that $1-\frac{1}{\ga}=\frac{1}{s}-\frac{1}{d}$, to get the desired estimate.

Assume then $V\in RH_q$, $\frac{d}{2}<q<d$. Let $w$ be such that $w^{s'}\in A_1^{\rho,\infty}$ and therefore $w^{s'}\in A_1^{\rho,\be}$, for some $\be\ge 0$. In this case it is also true that $w\in A_1^{\rho,\theta}$ with $\theta=\be/s'$.

Given $f\in L^1$, let us consider $P_j=P(x_j,r_j)$ the Calder\'on-Zigmund decomposition given in Lemma~\ref{lem:CZdecom} associated to $\theta$. We define the set of indexes
\begin{equation*}
J_1=\{j:\ r_j\leq \rho(x_j)\},\ \ J_2 = \{j:\ r_j> \rho(x_j)\},
\end{equation*}
and
\begin{equation*}
\Omega_1=\cup_{j\in J_1} P_j,\ \ \Omega_2=\cup_{j\in J_2} P_j.
\end{equation*}

Now we split $f=g+h+h'$, as
\begin{equation*}
g(x)=\begin{cases}
\frac{1}{|P_j|}\int_{P_j} f, & \text{if}\ x\in P_j,\ j\in J_1,\\
0, & \text{if}\ x\in P_j,\ j\in J_2,\\
f(x), & \text{if}\ x\notin \Omega,
\end{cases}
\end{equation*}
with $\Omega=\Omega_1 \cup \Omega_2$,
\begin{equation*}
h(x)=\begin{cases}
f(x)-\frac{1}{|P_j|}\int_{P_j} f, & \text{if}\ x\in P_j,\ j\in J_1,\\
0, & \text{otherwise,}
\end{cases}
\end{equation*}
and therefore $h'(x)=\chi_{\Omega_2}f$.

Let $\Pt_j=P_j(x_j,2r_j)$ and $\Omt = \cup_{j}\Pt_j$. Now,
\begin{equation}\label{wset}
w(\{x:\ |\rie_b f|(x)>\la\} \ \leq \ w(\Omt) + w(\{x\notin \Omt :\ |\rie_b f|(x)>\la\}.
\end{equation}

The first term of the last expression, can be controlled using \eqref{CZcub} and that $w\in A_1^{\rho,\theta}$ (see Remark~\ref{rem:Apcubos}), as
\begin{equation}
\begin{split}
w(\Omt)\ & \leq\ \sum_j w(\Pt_j)\ \lesssim\ \frac{1}{\la} \sum_j \frac{w(\Pt_j)}{|\Pt_j|} \left(1+\frac{r_j}{\rho(x_j)}\right)^{-\theta}\int_{P_j}|f| \\ 
& \lesssim\ \frac{1}{\la} \sum_j \inf_{P_j}w  \int_{P_j} |f|  \ \lesssim\ \frac{1}{\la}\int_{\RR^d}|f|\,w.
\end{split}
\end{equation}

For the second term of \eqref{wset}, it is enough to arrive to the right hand side of inequality \eqref{tipod} estimating $II_1\ =\ w(\{x:\ |\rie_b g(x)|>\la\})$, $II_2\ =\ w(\{x\notin\Omt:\ |\rie_b h(x)|>\la\})$ and $II_3\ = \ w(\{x\notin\Omt:\ |\rie_b h'(x)|>\la\})$.

To deal with $II_1$ notice that, from Lemma~\ref{lem:CZdecom} it follows that $|g|\leq \la$. On the other hand, from Theorem~\ref{teo:Lp} it follows that $\rie_b$ is bounded on $L^p(w)$ for some $p$ close enough to one.

In fact, from $w^{s'}\in A_1^{\rho,\infty}$ we get $w^{s'\nu}\in A_1^{\rho,\infty}$ for some $\nu>1$ (see Lemma~\ref{propo:Apmenos}) and taking $p$ such that $p(1-s')+s'=\frac{1}{\nu}$ it is easy to check that $w^{-\frac{1}{p-1}}\in A_{p'/s'}^{\rho,\infty}$. Therefore, since strong type implies weak type $(p,p)$, we get
\begin{equation}\label{II1}
\begin{split}
w(\{x:\ |\rie_b g(x)|>\la\}) \ \lesssim\ \frac{1}{\la^p}\int_{\RR^d}|g|^p\,w\ \lesssim \frac{1}{\la} \left(\sum_{j\in J_1} \frac{w(P_j)}{|P_j|}\int_{P_j}|f| + \int_{\Omega^c}|f|\,w\,\right).
\end{split}
\end{equation}
Since $w\in A_1^{\rho,\infty}$ and for $j\in J_1$, $r_j\leq \rho(x_j)$ we have $\frac{w(P_j)}{P_j}\lesssim \inf_{P_j} w$, and hence the last expression in \eqref{II1} can be easily bounded by $\frac{1}{\la}\int_{\RR^d}|f|\,w$.

To take care of $II_2$ we observe that
\begin{equation*}
\rie_b h(x)\ =\ \int_{\RR^d} \ke(x,y) [b(x)-b(y)] h(y)\, dy\ = \ \sum_{j\in J_1} \int_{P_j} \ke(x,y) [b(x)-b(y)] h(y)\, dy.
\end{equation*}
Adding and subtracting $b_{P_j}$ inside the integral we write
\begin{equation*}
\rie_b h(x)\ = \ - A(x) + B(x),
\end{equation*}
where
\begin{equation*}
A(x)\ = \sum_{j\in J_1} \int_{P_j} \ke(x,y) [b(y)-b_{P_j}] h(y)\, dy,
\end{equation*}
and
\begin{equation*}
B(x)\ = \sum_{j\in J_1} \int_{P_j} \ke(x,y) [b(x)-b_{P_j}] h(y)\, dy.
\end{equation*}

So we need to estimate
\begin{equation*}
II_{2,1}=w(\{x\notin\Omt:\ |A(x)|>\la\}\ \ \text{and}\ \ \ II_{2,2}=w(\{x\notin\Omt:\ |B(x)|>\la\}
\end{equation*}

To deal with the first expression let $\nu>1$ be such that $w^{s'\nu}\in A_1^{\rho,\infty}$. Hence, according to Remark~\ref{rem:delA1} there exists $\sigma\ge 0$ such that 
\begin{equation}\label{esA1}
M^\sigma(w^{s'\nu})\ \lesssim\ w^{s'\nu}.
\end{equation}

We set $w_* = w\chi_{\Omt^c}$ and since $w_*^{s'\nu}\in \Lloc$ we may apply Lemma~\ref{lem:MA1} for $g=w_*^{s'\nu}$, $\theta=\sigma$ and $\de=1/\nu$ to get that the weight $u=(M^{\sig}w_*^{s'\nu})^{\frac{1}{s'\nu}}$ is such that $u^{s'}$ belongs to $A_1^{\rho,\infty}$. Also, from differentiation, $w_* \leq u$ and from \eqref{esA1} $u\lesssim w$. Moreover, notice that for $y,\, z \in P_j$, $j\in J_1$ we have $u(x) \simeq u(y)$. This is due to the facts that $w_* = 0$ in $P_j$ and that $\rho(x)\simeq\rho(y)$.

Then since $A(x)=-\rie (\sum_{j\in J_1} (b-b_{P_j}) \chi_{P_j} h)(x)$ and $u^{s'}\in A_1^{\rho,\infty}$ (see Theorem~3 in \cite{BHS_schr-w}),
\begin{equation*}
\begin{split}
II_{2,1} \ &  = \ w_*(\{x:\ |A(x)|>\la \})\\
& \lesssim\ u(\{x:\ |A(x)|>\la \})\\
& \lesssim\ \frac{1}{\la} \sum_{j\in J_1} \int_{P_j} [b(y)-b_{P_j}] |h(y)| u(y)\, dy\\
& \lesssim\ \frac{1}{\la} \sum_{j\in J_1} \inf_{P_j} u \int_{P_j} [b(y)-b_{P_j}] |f(y)|\, dy\\
& \hspace{20pt} + \frac{1}{\la} \sum_{j\in J_1} \inf_{P_j} u \frac{1}{|P_j|}\int_{P_j} [b(y)-b_{P_j}] \int_{P_j} |f(y)|\, dy.
\end{split}
\end{equation*}
Clearly, the last sum is controlled by $[b]_\theta\,\|fw\|_1$ since $u\leq w$. For the first term we apply H\"older's inequality \eqref{Holderfi} and Lemma~\ref{lem:JN},
\begin{equation*}
II_{2,1}\ \lesssim \ \frac{1}{\la} [b]_\theta \sum_{j\in J_1} \inf_{P_j} u |P_j| \|f\|_{\varphi,P_j}.
\end{equation*}

We remind that from \cite{MR0126722}, p. 92, we get that for any cube $Q$
\begin{equation*}
\|f\|_{\varphi,Q} \simeq \inf_{t>0}\left\{t+\frac{t}{|Q|}\int_Q \varphi\left(\frac{|f|}{t}\right)\right\}.
\end{equation*}

Now, taking $t=\la$,
\begin{equation}\label{conphi}
\frac{|P_j|}{\la}\|f\|_{\varphi,P_j} \ \lesssim\ |P_j|+\int_{P_j}\varphi\left(\frac{|f|}{\la}\right).
\end{equation}

But since $P_j$ satisfies \eqref{CZcub}, we have
\begin{equation}\label{condcub}
|P_j|\lesssim \frac{1}{\la} \int_{P_j} |f|.
\end{equation}

Inserting these estimates an using again that $u\lesssim w$ it follows
\begin{equation*}
II_{2,1}\ \lesssim\ [b]_\theta\ \left(\int_{\RR^d} \frac{|f|}{\la} w +  \int_{\RR^d} \varphi\left(\frac{|f|}{\la}\right) w \right).
\end{equation*}

For $II_{2,2}$ we apply Tchebycheff inequality to get
\begin{equation*}
\begin{split}
II_{2,2} \ & \lesssim\ \frac{1}{\la} \int_{\Omt^c} |B|w\\
& \lesssim\ \frac{1}{\la} \sum_{j\in J_1} \int_{\Pt_j^c} |b(x)-b_{P_j}| \left(\int_{P_j} |\ke(x,y)-\ke(x,x_j)|\,|h(y)|\, dy\right) w(x)\,dx \\
& \lesssim\ \frac{1}{\la} \sum_{j\in J_1} \int_{P_j} |h(y)| \left(\int_{\Pt_j^c} |b(x)-b_{P_j}|\,|\ke(x,y)-\ke(x,x_j)|\,w(x)\,dx\right) dy.
\end{split}
\end{equation*}
The inner integrals may be estimate splitting into annuli and applying H\"older's inequality with $s$, $s'\nu$, $\ga$ with $\nu>1$ such that $w^{s'\nu}\in A_1^{\rho,\infty}$ and $\frac{1}{s} + \frac{1}{s'\nu} + \frac{1}{\ga} = 1$. In this way, setting $P_j^k=2^kP_j$ we have
\begin{equation}\label{cruz}
\begin{split}
\int_{\Pt_j^c} & |b(x)-b_{P_j}|\,|\ke(x,y)-\ke(x,x_j)|\,w(x)\,dx \\
& \lesssim \ \sum_{k=2}^\infty \left(\int_{P_j^k} |b(x)-b_{P_j}|^{\ga}\,dx\right)^{1/\ga}\\
& \hspace{40pt}\times\left(\int_{P_j^k\setminus P_j^{k-1}} |\ke(x,y)-\ke(x,x_j)|^s\,dx\right)^{1/s} \left(\int_{P_j^k}w^{s'\nu}\right)^{\frac{1}{s'\nu}}.
\end{split}
\end{equation}

Next, observe that if $b\in BMO^\theta_\infty(\rho)$, using Lemma~\ref{lem:JN}, for some $\eta\ge \theta$ we have
\begin{equation*}
\begin{split}
& \left(\int_{P_j^k} |b(x)-b_{P_j}|^\ga\,dx\right)^{1/\ga} \\ 
& \hspace{30pt} \lesssim \ \left(\int_{P_j^k} |b(x)-b_{P_j^k}|^{\ga}\,dx\right)^{1/\ga} + |P_j^k|^{1/\ga} \sum_{i=0}^{k-1} \frac{1}{|P_j^i|} \int_{P_j^i} |b(x)-b_{P_j^i}|\\
& \hspace{30pt} \lesssim \ [b]_\theta |P_j^k|^{1/\ga} \left[ \left(1 + \frac{2^kr_j}{\rho(x_j)}\right)^{\eta} + \sum_{i=0}^{k-1} \left(1 + \frac{2^ir_j}{\rho(x_j)}\right)^{\theta}\right]\\
& \hspace{30pt} \lesssim \ k [b]_\theta |P_j^k|^{1/\ga} \left(1 + \frac{2^kr_j}{\rho(x_j)}\right)^{\eta}.
\end{split}
\end{equation*}
Also, since $w^{s'\nu}\in A_1^{\rho,\infty}$, for some $\sigma>0$ we have
\begin{equation}\label{punto}
\left(\int_{P_j^k}w^{s'\nu}\right)^{\frac{1}{s'\nu}} \ \lesssim \ \inf_{P_j}w\, |P_j^k|^{1/s'\nu} \left(1 + \frac{2^kr_j}{\rho(x_j)}\right)^{\sigma}.
\end{equation}

Therefore, since $\frac{1}{s'\nu} + \frac{1}{\ga} = \frac{1}{s'}$ and $P_j\subset P_j^k$, the right hand side of \eqref{cruz} can be bounded by a constant times 
\begin{equation*}
\begin{split}
[b]_\theta \inf_{P_j}w \sum_{k=2}^{\infty} k(2^kr_j)^{d/s'} \left(1 + \frac{2^ir_j}{\rho(x_j)}\right)^{\eta+\sigma}\left(\int_{P_j^k\setminus P_j^{k-1}} |\ke(x,y)-\ke(x,x_j)|^s\,dx\right)^{1/s}
\end{split}
\end{equation*}
but for $P_j$, we have $|y-x_j|<r_j$ so we may apply H\"ormander's type condition of Lemma~\ref{lem:hormander}. Therefore,
\begin{equation*}
\begin{split}
II_{2,2}\ \lesssim\ \frac{[b]_\theta}{\la} \sum_{j\in J_1} \inf_{P_j}w \int_{P_j} |h|\ \lesssim\ \frac{1}{\la} \sum_{j\in J_1} \int_{P_j}|f|w \ \lesssim\ \frac{1}{\la} \int_{\RR^d} |f|w.
\end{split}
\end{equation*}

Finally, we take care of $III$ which involves $h'=f\chi_{\Om_2}$. By Tchebycheff inequality and proceeding as for $II_{2,2}$,
\begin{equation*}
\begin{split}
III\ \lesssim\ \frac{1}{\la} \sum_{j\in J_2} \int_{P_j} |f(y)| \left(\int_{\Pt_j^c} |b(x)-b_{P_j}|\,|\ke(x,y)|\,w(x)\,dx\right) dy.
\end{split}
\end{equation*}

Now, for each $j\in J_2$ we bound the inner integral splitting into annuli and applying H\"older's inequality as in \eqref{cruz}. With the same notation there we get using Lemma~\ref{lem:kcor},
\begin{equation}\label{40}
\begin{split}
\int_{\Pt_j^c} & |b(x)-b(y)|\,|\ke(x,y)|\,w(x)\,dx \\
& \lesssim \ \sum_{k=2}^\infty (2^kr_j)^{-1-\frac{d}{q'}} \left(\frac{\rho(x_j)}{2^kr_j}\right)^{N-\mu d} \left(\int_{P_j^k} |b(x)-b(y)|^{\ga}\,dx\right)^{1/\ga} \left(\int_{P_j^k}w^{s'\nu}\right)^{\frac{1}{s'\nu}}.
\end{split}
\end{equation}
For the factor with $w$ we use estimate \eqref{punto}, and the one concerning $b$ we can add and subtract $b_{P_j^k}$ to obtain
\begin{equation*}
\begin{split}
& \left(\int_{P_j^k} |b(x)-b_{P_j}|\,dx\right)^{1/s} \ \lesssim \ (2^kr_j)^{-1-\frac{d}{q'}} \left[[b]_\theta\left(\frac{2^kr_j}{\rho(x_j)}\right)^\eta + |b_{P_j^k} - b(y)|\right].
\end{split}
\end{equation*}

Collecting estimates, setting $\al=N-\mu d -\eta-\sig$ and using that $r_j\ge \rho(x_j)$ for $j\in J_2$, we get
\begin{equation}\label{ultima}
\begin{split}
& III\ \lesssim \\
&\ \ \frac{1}{\la} \sum_k 2^{-k\al} \sum_{j\in J_2} \left(\frac{r_j}{\rho(x_j)}\right)^{\al} \inf_{P_j} w \int_{P_j} |f(y)| \left[[b]_\theta\left(\frac{2^kr_j}{\rho(x_j)}\right)^\eta + |b_{P_j^k} - b(y)|\right]\,dy.
\end{split}
\end{equation}

For the term with $[b]_\theta$ choosing $N$ such that $N-\mu d-\eta-\sig>0$ and using that $r_j\ge \rho(x_j)$ for $j\in J_2$, to obtain that it is bounded by a constant times $\frac{1}{\la}\int f w$.

For the other term we apply as before H\"older with $\ff$ and $\fft$ to get
\begin{equation*}
\int_{P_j} |f(y)| |b_{P_j^k} - b(y)| \,dy\ \lesssim\ |P_j|\, \|f(y)\|_{\ff,P_j}\, \|b_{P_j^k} - b\|_{\tilde{\ff},P_j}.
\end{equation*}

Then, we apply Lemma~\ref{lem:JN} to bound the last factor. For the first factor we use \eqref{conphi} and \eqref{condcub}. Therefore, choosing $N$ large enough in \eqref{ultima} such that $N-\mu d-\eta-\sig-M>0$ we obtain that expression bounded by
\begin{equation*}
\int_{P_j} \ff\left(\frac{f}{\la}\right)\,w.
\end{equation*}

\end{proof}

\begin{rem}
We want to point out that inequality \eqref{tipod} is also true for $\rie^*_b$ with weights in $A_1^{\rho,\infty}$ provided the potential $V$ belongs to $RH_d$. On the other hand, when $V\in RH_q$ for some $q>d/2$ but $V\notin RH_d$, we can not expect this kind of result for $\rie^*_b$ since $\rie^*$ is not of weak type $(1,1)$ for $w=1$ (see \cite{shen}). Therefore, in order to get \eqref{tipod} for $\rie^*_b$ when $V\in RH_d$ we can not argue as we did for $\rie_b$ in that case. Nevertheless, a close look at the proof of the case $q<d$ reveals that the same pattern could be followed in this case.

In fact, notice that the only instances in the argument where we use properties of $\rie_b$, $\rie$ or of the kernel $\ke$ are the following:
\begin{enumerate}[(i)]
\item \label{st} Strong type $(p,p)$ of $\rie_b$ with the weight $w$ for some $p>1$ (see \eqref{II1}).
\item \label{wt} Weak type $(1,1)$ of $\rie$ with the weight $u=(M^\sigma w_*^{\nu s'})^{\frac{1}{\nu s'}}$, when estimating $II_{2,1}$.   
\item \label{hor} H\"ormander's like property of $\ke$ (see \eqref{hormander}) to bound $II_{2,2}$.
\item \label{estk} Estimates of the size of $\ke$ given by Lemma~\ref{lem:kcor} to obtain inequality \eqref{40}.
\end{enumerate}

When $V\in RH_d$ all these properties are true for $\rie^*_b$, $\rie^*$ and $\ke^*$ for the corresponding value $s=\infty$. In fact, that \eqref{st} and \eqref{wt} are true is a consequence of Theorem~\ref{teo:Lp} of Section~\ref{sec:lpineq} and Theorem~3 in \cite{BHS_schr-w} together with Lemma~\ref{lem:p1} above.

Regarding \eqref{hor} and \eqref{estk} it is known that $\ke^*$ is a Calder\'on-Zygmund kernel when $V\in RH_q$ and moreover it satisfies the stronger inequalities
\begin{equation*}
|\ke^*(x,y)|\ \leq\ C_N\left(1+\frac{|x-y|}{\rho(x)}\right)^{-N} \frac{1}{|x-y|^d},
\end{equation*}
and
\begin{equation*}
|\ke^*(x,y)-\ke^*(x,z)| \ \leq C_N\,\left(1+\frac{|x-y|}{\rho(x)}\right)^{-N} \frac{|y-z|^\de}{|x-y|^{d+\de}},
\end{equation*}
whenever $2|y-z|\le |x-y|$, for some $\de>0$ and any $N\ge 0$. Also, $\rho(x)$ can be substituted by $\rho(y)$ in all instances (see Lemma~4 in \cite{BHS_commut}).
\end{rem}

\def\cprime{$'$}

\end{document}